\documentclass[11pt,reqno]{amsproc}
\usepackage[margin=1in]{geometry}
\usepackage{amsmath, amsthm, amssymb}
\usepackage{times, esint, stackrel, enumitem, color, hyperref}
\usepackage[displaymath, mathlines]{lineno}
\usepackage{setspace}

\newtheorem{thm}{Theorem}
\newtheorem{prop}{Proposition}
\newtheorem{lemma}{Lemma}

\newtheorem{rem}{Remark}

\newcommand{\be}{\begin{equation}}
\newcommand{\ee}{\end{equation}}
\newcommand{\bea}{\begin{eqnarray}}
\newcommand{\eea}{\end{eqnarray}}
\newcommand{\ve}{{\varepsilon}}
\newcommand{\vphi}{{\varphi}}
\newcommand{\rmd}{{\rm d}}
\newcommand{\bx}{{ {x} }}
\newcommand{\br}{{{r}}}
\newcommand{\bv}{{ {v}}}
\newcommand{\bX}{{X}}
\newcommand{\bu}{{{u}}}
\newcommand{\ol}[1]{\mkern 1.5mu\overline{\mkern-1.5mu#1\mkern-1.5mu}\mkern 1.5mu}

\newcommand{\Dlim}{{\mathcal{D}'}\mbox{-}\lim}

\title[Cascades and Time-Asymmetry]{ Turbulent Cascade Direction and Lagrangian Time-Asymmetry}
\author{Theodore D. Drivas}
\address{Department of Mathematics, Princeton University, Princeton, NJ 08544}
\email{tdrivas@math.princeton.edu}
\date{today}

\begin{document}

\begin{abstract}\vspace{-1mm}
We establish Lagrangian formulae for energy conservation anomalies involving the discrepancy between short-time two-particle dispersion forward and backward in time.  
These results are facilitated by a rigorous version of the Ott-Mann-Gaw\c{e}dzki relation, sometimes described as a ``Lagrangian analogue of the $4/5$--law".   In particular, we prove that for weak solutions of the Euler equations, the Lagrangian forward/backward dispersion measure matches on to the energy defect \cite{O49,DR00} in the sense of distributions.    For strong limits of $d\geq3$ dimensional Navier-Stokes solutions the defect distribution coincides with the viscous dissipation anomaly.  The Lagrangian formula shows that particles released into a $3d$ turbulent flow will initially disperse faster backward--in--time than forward, in agreement with recent theoretical predictions of Jucha et.~al \cite{Jucha14}.  In two dimensions,  we consider strong limits of solutions of the forced Euler equations with increasingly high-wavenumber forcing as a model of an ideal inverse cascade regime.  We show that the same Lagrangian dispersion measure matches onto the anomalous input from the infinite-frequency force. As forcing typically acts as an energy source, this leads to the prediction that particles in $2d$ typically disperse faster forward in time than backward, which is opposite to what occurs in $3d$.  Time-asymmetry of the Lagrangian dispersion is thereby closely tied to the direction of the turbulent cascade, downscale in $d\geq 3$ and upscale in $d=2$. These conclusions lend support to the conjecture of \cite{ED15} that a similar connection holds for time-asymmetry of Richardson two-particle dispersion and cascade direction, albeit at longer times.
\end{abstract}
 
\maketitle
\vspace{-1mm}

\section{Introduction}

 Perhaps the most notable difference between $2d$ and $3d$ incompressible turbulence is the direction of the energy cascade. In three-dimensions, fluid energy is typically transferred from large to small scales via a non-linear process called the \emph{direct} turbulent cascade.  Serving as a sink at the end of this cascade is molecular viscosity $\nu$, which acts to dissipated the kinetic energy deposited at small scales.
Remarkably, observations from experiments and simulations of  forced or freely decaying turbulence show that, in the limit of high Reynolds number (equivalently zero viscosity), the kinetic energy dissipation is non-zero
 \be\label{zerothLaw3d}
  \lim_{\nu\to 0}  \nu\langle |\nabla u|^2\rangle = \varepsilon>0
 \ee
 where  $\langle \cdot \rangle$ is some relevant averaging procedure, space, time or ensemble. This is the so-called ``zeroth law" of turbulence and is often referred to as \emph{anomalous dissipation} and is the central postulate of the celebrated Kolmogorov 1941 (K41) theory \cite{K41}.    This property reflects the fact that in three-dimensions the turbulent cascade is exceptionally effective at transferring energy from large to small scales.  Indeed, in the high--$Re$ limit where the viscous length scale vanishes, the cascade process continues to transport appreciable amounts of energy to scales where it can be effectively dissipated by infinitesimal viscosity \cite{O49,K41,Taylor17}. 

In two-dimensional incompressible turbulence on domains without boundary, viscosity plays a negligible role for the energy budget. For smooth initial data, the viscous energy dissipation at finite times always tends to zero as $\nu\to 0$  \cite{Taylor17,Taylor15}.  It has long been recognized that the source of major differences between $d=2$ and $d\geq3$ is the presence of an additional invariant -- the enstrophy, see e.g. \cite{Lee51,Fjor53,E96}.   Kraichnan \cite{Kr67} (see also \cite{Leith68,B69}) proposed that this extra constraint results in two simultaneous inertial ranges in the flow, an inverse energy cascade range and a forward enstrophy cascade. In the inverse energy cascade range, the energy input by forcing is transported from small to large scales -- in contrast with $3d$ -- where it accumulates until it is depleted, for example, by linear damping or (ineffectually) by viscosity. 

As a model for the inverse energy cascade, consider forced two-dimensional fluid without viscosity with sufficiently smooth initial data.  In this setting, the energy of the fluid can change only due to the input from the forcing. It is well known that typically the energy input undergoes a direct cascade in scales smaller  than the characteristic forcing scale $\ell_f$, while it undergoes an inverse cascade for scales $\ell	 \gtrsim  \ell_f$.    Thus, an extended inverse cascade range can be achieved by using a high-wavenumber forcing, spectrally concentrated around wavenumber $k_f\sim  2\pi/\ell_f$, and considering the limit $k_f\to \infty$. This produces an  ``infinite frequency" forcing which vanishes in the sense of distributions but may continue to input (or remove) energy 
 \be\label{zerothLaw2d}
  \lim_{k_f\to \infty}  \langle u\cdot f \rangle = I\neq0.
 \ee
When \eqref{zerothLaw2d} holds, it is analogous to the dissipative anomaly \eqref{zerothLaw3d} in higher dimensions.
Whenever the forcing acts as an energy source, which is typical, $I$ is positive and we call \eqref{zerothLaw2d} the \emph{anomalous input} or \emph{production anomaly}.

Despite being an intrinsically Eulerian object, the dissipative anomaly  \eqref{zerothLaw3d}  is known to have some interesting connections with Lagrangian aspects of turbulence.   In 1926, Richardson \cite{Rich26} predicted that particles pairs in the inertial range of a high--$Re$ turbulent flow have mean-squared separation that grows as $t^3$, i.e.
\be\label{Rich}
\langle |\delta\bX_{t_0,t}(\br;\bx)|^2\rangle \sim g\varepsilon t^3
\ee
where  $\delta\bX_{t_0,t}(\br;\bx):= \bX_{t_0,t}(\bx+\br)-  \bX_{t_0,t}(\bx)$ is the Lagrangian deviation, $g$ is the Richardson constant, and the tracers particles $X_{t_0,t}(x)$ satisfy
 \be
 \frac{d}{dt} X_{t_0,t}(x) ={u}( X_{t_0,t}(x),t), \qquad  X_{t_0,t_0}(x)= x.
 \ee
Richardson's prediction \eqref{Rich} notably involves the viscous dissipation rate $\varepsilon$ which remains finite in the zero-viscosity limit \eqref{zerothLaw3d}.   Somewhat mysteriously,  Richardson dispersion is observed numerically to be faster backward-in-time than forward for $3d$ turbulence and faster forward-in-time than backward in $2d$ \cite{Faber09,Saw05}.  This observation, as well as insight from toy models, led to the conjecture of \cite{ED15} that the direction of the cascade --  inverse or direct -- and time-asymmetry of Lagrangian particle dispersion are closely related.  

Recent work on mean-squared particle dispersion has shed new light on Lagrangian manifestations of time asymmetry and its connection to the turbulent cascade.  See \cite{Jucha14,Bitane12,Falk13} and also the recent review \cite{Xu16}.  These studies employ the so-called \emph{Ott-Mann-Gaw\c{e}dzki} relation \cite{OttMan05,Falk01}, sometimes described as the ``Lagrangian analog of the $4/5$--law", in order to obtain an explicit short-time expansion for the two-particle dispersion in terms of purely Eulerian quantities.  For  inertial range separations $r$, this relation states:
\begin{equation}\label{OttMann0}
\left.\frac{1}{2} \frac{d}{d \tau} \Big\langle |\delta_r \bv(\tau;\bx,t) |^2\Big\rangle\right|_{\tau=0} \simeq -2  \varepsilon 
\end{equation}
with the Lagrangian velocity $v(\tau,x;t):=\bu(\bX_{t,t+\tau} (\bx),t+\tau)$ and $\delta_r  \bv(\tau;\bx,t):=  v(\tau,x+r;t)-v(\tau,x;t)$.  Standard derivations of the relationship assume spatial isotropy and the average must be either interpreted as over the spatial domain, or as a time/ensemble average provided the fields are homogenous.   

With the Ott-Mann-Gaw\c{e}dzki relation in hand, the relative mean--squared dispersion of Lagrangian tracers for short-times can be calculated using only local (in time) Eulerian quantities in closed form \cite{Bitane12}:
\begin{align} \label{reldispFG}
\Big\langle | \delta \bX_{t,t+\tau}(\br; \bx)-&\br|^2\Big\rangle \approx   S_2^{\bu}(\br,t)  \tau^2 - 2  \varepsilon \tau^3 +{O} (\tau^4)
\end{align}
where $S_2^{{\bu}}(\br,t)  := \langle |\delta \bu(\br,t)|^2\rangle$ is the second-order structure function. In simulations of three-dimensional turbulence, the leading order quadratic and cubic behavior for time differences of order the local turnover time at scale $|r|$ is verified \cite{Bitane12}.      Note that, although the energy dissipation rate $\ve$ appears as coefficients in both cubic terms, the $\tau^3$ term in \eqref{reldispFG} is for short times only and is not the same as the behavior that Richardson predicted \eqref{Rich} which holds at later times.

Recently, Jucha et. al. \cite{Jucha14} realized that for $3d$ turbulent flows,  Eq. \eqref{reldispFG} can be used to predict that pairs of Lagrangian particles initially spread faster backward-in-time than forward-in-time.  This is deduced by inspecting the behavior of \eqref{reldispFG} under time reversal $\tau\to -\tau$ and noting that the ${O}(\tau^2)$ term is invariant whereas the ${O}(\tau^3)$ term changes sign.  Since $ \varepsilon > 0$ for high-Reynolds number 3$d$ turbulence \eqref{zerothLaw3d}, this  ${O}(\tau^3)$ term tends to enhance the dispersion backwards-in-time and deplete the dispersion forwards-in-time, thereby establishing a Lagrangian ``arrow of time".  Unfortunately, at high-Reynolds numbers, the realm of validity of the expansion \eqref{reldispFG} becomes vanishing small. In particular, the Taylor series expansion of the particle trajectories used to derive \eqref{reldispFG} is only guaranteed to converge in a neighborhood of times on the order of the Kolmogorov time-scale $\tau_\eta \sim (\nu/\ve)^{1/2}$. Thus, it is desirable to have an alternative Lagrangian measure of time--asymmetry that remains valid for arbitrarily large Reynolds numbers. 

We prove here that there is such a Lagrangian measure involving the short-time dispersion of tracer particles in \emph{coarse-grained}  (or mollified) fields $\ol{u}_\ell$ instead of their fine-grained counterparts $u$.  In particular, such trajectories satisfy
 \be\label{eqn:traj}
 \frac{d}{dt} X_{t_0,t}^{\ell}(x) =\ol{u}_\ell( X_{t_0,t}^{\ell}(x),t), \qquad  X_{t_0,t_0}^{\ell}(x)= x
 \ee
 with  $\ol{g}_\ell = \int_{\mathbb{T}^d}G_\ell(r) g(x+r) dr$  for any $g\in L^1(\mathbb{T}^d)$ where $G\in C_0^\infty(\mathbb{T}^d)$  is a standard mollifier, compactly supported in the unit ball, and $G_\ell(r) = \ell^{-d} G(r/\ell)$.    Then, the following are novel formulae for the dissipation/input anomalies which are purely Lagrangian in nature (albeit, for asymptotically short-times).

 \begin{thm}\label{theorem} Fix an standard mollifier $\psi\in C_0^\infty(\mathbb{T}^d)$  with  ${\rm supp} (\psi)\subseteq B_1(0)$ and $\psi_R(r)  = R^{-d} \psi(r/R)$ and denote 
$
 \langle f(r) \rangle_R:= \int_{\mathbb{T}^d} f(r) \psi_R(r)dr
$
 for any $f\in L^1( \mathbb{T}^d)$.    \vspace{1mm}
 \begin{enumerate}[label=(\roman*)]
  \item Let $d\geq 2$ and let $u\in  {L^\infty(0,T;L^2(\mathbb{T}^d))}\cap L^3(0,T;L^3(\mathbb{T}^d))$ be any weak solution to the Euler equations with initial data $u_0\in L^2(\mathbb{T}^d)$ and forcing $f\in L^{ {\infty}}(0,T;L^2(\mathbb{T}^d))$.  Then 
   \be\label{Eulerasym}
\lim_{R \to 0}\lim_{\ell \to 0}  \lim_{\tau\to 0}\frac{\langle |\delta X_{t,t+\tau}^{\ell}(r;x)-r|^2\rangle_R - \langle |\delta X_{t,t-\tau}^{\ell}(r;x)-r|^2\rangle_R}{4\tau^3} = -\Pi[\bu]
 \ee
   in the sense of distributions in $\mathbb{T}^d\times[0,T]$,  where the conservation anomaly $\Pi[u]$ is defined by \eqref{EulerBalanceWeak}. \vspace{-2mm}
   
 \item Let $d\geq 3$ and $u^\nu\in L^\infty(0,T;L^2({\mathbb T}^d))\cap L^2(0,T; H^1({\mathbb T}^d))$ be any Leray solutions to the Navier-Stokes equations with initial data $u_0^\nu\in L^2(\mathbb{T}^d)$ and forcing $f^\nu\in L^{ {\infty}}(0,T;L^2(\mathbb{T}^d))$ (with norms uniformly bounded).  Assume  $u^\nu \to u$ strongly in $L^3(0,T;L^3(\mathbb{T}^d))$. Then 
 \be\label{3dasym}
\lim_{R \to 0}\lim_{\ell \to 0}  \lim_{\tau\to 0} \lim_{\nu\to 0} \frac{\langle |\delta X_{t,t+\tau}^{\ell,\nu}(r;x)-r|^2\rangle_R - \langle |\delta X_{t,t-\tau}^{\ell,\nu}(r;x)-r|^2\rangle_R}{4\tau^3} = -\varepsilon[\bu] 
 \ee
  in the sense of distributions in $\mathbb{T}^d\times[0,T]$,  where the dissipative anomaly $\varepsilon[\bu]$ is defined by \eqref{NSLimitBalanceWeak}.  \vspace{3mm}
  
 \item Let $d=2$ and $u^{k_f}\in C([0,\infty); W^{1,r}(\mathbb{T}^2) )$, $r> 3/2$ be any weak solutions to the Euler equations with initial data $u_0^{k_f}\in L^2(\mathbb{T}^2)$ and  $\omega_0^{k_f}:=\nabla^\perp\cdot u_0^{k_f}\in L^r(\mathbb{T}^2)$ and forcing $f^{k_f}\in L^{ {\infty}}([0,\infty);L^2((\mathbb{T}^2))$ such that $f^{k_f}\to 0$ as $k_f\to \infty$ in sense of distributions.   Assume $u^{k_f} \to u$ strongly in $L^3(0,T;L^3(\mathbb{T}^2))$. Then 
 \be\label{2dasym}
\lim_{R \to 0}\lim_{\ell \to 0}  \lim_{\tau\to 0} \lim_{k_f\to \infty} \frac{\langle |\delta X_{t,t+\tau}^{\ell,k_f}(r;x)-r|^2\rangle_R - \langle |\delta X_{t,t-\tau}^{\ell,k_f}(r;x)-r|^2\rangle_R}{4\tau^3} = I[\bu] 
 \ee
 in the sense of distributions in $\mathbb{T}^2\times[0,T]$, where the production anomaly $I[\bu]$ is defined by \eqref{ForcedLimitBalanceWeak}.
 \end{enumerate}
 \end{thm}

 The expressions \eqref{Eulerasym}, \eqref{3dasym} and \eqref{2dasym} represent Lagrangian formulae for conservation law anomalies which are local in space and time.  These expressions involve computing the difference of short time dispersion both forward and backward in time\footnote{We remark that  Frishman \& Falkovich \cite{Frish14} have argued on theoretical grounds that, unlike Eq. \eqref{reldispFG}, the short-time expansion of the difference of forward/backward dispersion appearing in \eqref{3dasym} for incompressible Navier-Stokes should have a finite radius of convergence at a fine-grained level ($\ell\equiv 0$), even in the limit of $\nu\to 0$.  This remarkable property may be useful to bridge the gap between the asymptotically short time results presented here and the observations of Richardson dispersion at later times.}, and averaging over particle pairs in a small region of size $R$ with a kernel $\psi$.  Note, however, that the resulting distributions are independent of choice of this kernel.  The main physical interest of Theorem \ref{theorem} is that, the asymmetry in the short-time dispersion precisely correlates with the turbulent cascade direction (at small scales), i.e. the sign of the flux.   Part  (i) of the Theorem applies to  general weak Euler solutions for which the cascade is characterized by the distribution $\Pi[u]$ and may occur with either sign. Part (ii) is concerned with the inviscid limit of Navier-Stokes solutions; it provides rigorous mathematical justification of the observations of Jucha et al. \cite{Jucha14} and shows that without need for ensemble averaging or any assumption of isotropy or homogeneity.      Part (iii) of the Theorem extends these considerations to the setting of $2d$ Euler in the limit of infinitely small-scale forcing, which serves as a model for the inverse energy cascade.    Unlike the situation in $3d$ turbulence, in this setting particles initially disperse faster forward-in-time than backward since typically forcing inputs energy  ($I>0$).  Thus, the information on time--asymmetry of short-time Lagrangian dispersion provided by  Theorem \ref{theorem} (ii) \& (iii) mirrors the observations of Richardson dispersion \cite{Faber09,Saw05} and lends support to the conjecture of \cite{ED15}.  
 
  {Physically, we never really go to the limit of zero time $\tau$, viscosity $\nu$, forcing scale $\ell_f$, filter scale $\ell$ or radial resolution $R$ which are technically required for the Theorem \ref{theorem}. In practice, our results should hold approximately within a range of scales that we now describe.  Assume that the following plausible bound on the $o(\tau^3)$ corrections in the equation for the difference of \eqref{reldispF} and \eqref{reldispB}  holds
\be\label{approxThm}
\Big\langle |\delta X_{t,t+\tau}^{\ell}(r;x)-r|^2\Big\rangle_{R,\varphi} - \Big\langle |\delta X_{t,t-\tau}^{\ell}(r;x)-r|^2\Big\rangle_{R,\varphi} = \langle\Pi_\ell\rangle_\vphi  \tau^3\left[1+ O\left(\frac{\tau}{\tau_\ell}\right)\right],
\ee
where $\tau_\ell= O(\ell/\delta u(\ell))$ is the local eddy turnover time at scale $\ell$, $\delta u(\ell)$ is some measure of the typical velocity fluctuation at that scale, and $\langle \cdot \rangle_\varphi$ is a $\varphi$--weighted space-time average where $\varphi\in C_0^\infty([0,T]\times \mathbb{T}^d)$. The hypothetical bound \eqref{approxThm} is the assertion that the Taylor series in time for trajectories in coarse-grained fields \eqref{eqn:traj} is valid until the local turnover time.  Then the results \eqref{3dasym} and \eqref{2dasym} hold for $\ell$, $R$ and $\tau$ in the ranges
 \begin{align}\label{scalerange1}
d\geq 3&:\qquad \qquad  \phantom{\ell_\nu\ll}\  \ell_\nu\ll \ell\ll R\ll L\quad{\rm and}\quad  \tau \ll \tau_\ell,\\
d=2&:\qquad \qquad \ell_\nu \ll \ell_f \ll  \ell\ll R\ll L \quad {\rm and} \quad  \tau \ll \tau_\ell, \label{scalerange2}
 \end{align} 
 where $\ell_\nu$ is the dissipative cutoff scale (in K41 theory, $\ell_\nu/L\sim Re^{-3/4}$) and $L$ is the integral scale (e.g. a characteristic length-scale  of the large-scale production mechanism).  The scale ranges \eqref{scalerange1} and \eqref{scalerange2} show that our results require a long inertial range with a large separation of scales to hold.  However, the studies  \cite{Jucha14,Bitane12} suggest that in three-dimensions fine-grained analogues of \eqref{approxThm}  hold to a reasonably degree of accuracy even at moderately large Reynolds number.  Improved accuracy as well as the rate of convergence to asymptotia can be investigated numerically.  
 }

The main technical tool used in the proof of Theorem \ref{theorem} is a generalization of the Ott-Mann-Gaw\c{e}dzki relation for particles moving in a coarse-grained fluid velocity field, which may be of independent interest.

 \begin{lemma}[Generalized Ott-Mann-Gaw\c{e}dzki Relation]\label{lemma} Let $\psi_R$ be as in Theorem \ref{theorem}.  \vspace{1mm}
\begin{enumerate}[label=(\roman*)]
\item  ($d\geq 2$) Under conditions of Theorem \ref{theorem}, (i), in the sense of distributions in $\mathbb{T}^d\times[0,T]$, we have
 \be\label{OttMannEuler}
\lim_{R \to 0}\lim_{\ell \to 0} \frac{1}{2} \frac{d}{d\tau} \langle |\delta_r v^{\ell}(\tau;x,t)|^2\rangle_R\Big|_{\tau=0} = -2\Pi[\bu]
 \ee
  where $v^{\ell}(\tau,x;t):= \ol{u}_\ell (X_{t,t+\tau}^{\ell}(x),t+\tau)$ and $\delta_r v^{\ell}(\tau;x,t)=  v^{\ell}(\tau,x+r;t)-v^{\ell}(\tau,x;t)$.
   \vspace{1mm}
   
\item  ($d\geq 3$) Under conditions of Theorem \ref{theorem}, (ii), in the sense of distributions in $\mathbb{T}^d\times[0,T]$, we have
 \be\label{OttMann3d}
\lim_{R \to 0}\lim_{\ell \to 0} \lim_{\nu \to 0} \frac{1}{2} \frac{d}{d\tau} \langle |\delta_r v^{\ell,\nu}(\tau;x,t)|^2\rangle_R\Big|_{\tau=0} = -2\varepsilon[\bu]
 \ee
  where $v^{\ell,\nu}(\tau,x;t):= \ol{(u^\nu)}_\ell (X_{t,t+\tau}^{\ell,\nu}(x),t+\tau)$. \vspace{1mm}
 
  \item  ($d= 2$) Under conditions of Theorem \ref{theorem}, (iii), in the sense of distributions in $ \mathbb{T}^2\times[0,T]$ we have
  \be\label{OttMann2d}
\lim_{R \to 0}\lim_{\ell \to 0} \lim_{k_f \to \infty} \frac{1}{2} \frac{d}{d\tau} \langle |\delta_r v^{\ell,k_f}(\tau;x,t)|^2\rangle_R\Big|_{\tau=0} = 2I[\bu]
 \ee
  where $v^{\ell,k_f}(\tau,x;t):= \ol{(u^{k_f})}_\ell (X_{t,t+\tau}^{\ell,k_f}(x),t+\tau)$. 
   \end{enumerate}
 \end{lemma}

That the energy flux-through-scale should appears in the Ott-Mann-Gaw\c{e}dzki relation \eqref{OttMannEuler} for Euler solutions was already essentially understood in \cite{Falk13,Falk01}.   Lemma \ref{lemma} is a precise mathematical formulation of this observation.  The formulae \eqref{OttMannEuler}, \eqref{OttMann3d}, and \eqref{OttMann2d} are independent of the $r$--averaging kernel $\psi\in C_0^\infty(\mathbb{T}^d)$.

In Section \ref{sec:anom} below, we describe an appropriate mathematical framework for describing dissipation/input anomalies \eqref{zerothLaw3d} and \eqref{zerothLaw2d}, as well as their connection to the turbulence cascade.  Specifically, in \S \ref{sec:3danom}, we review previous work  of Duchon \& Robert \cite{DR00} for weak solutions of the $3d$ Navier-Stokes equations and in \S \ref{sec:2danom}, we extend the work of \cite{DR00} to a small-scale forced $2d$ Euler setup which is proposed as a mathematical model for the study of the inverse cascade. 
Proofs are deferred to Section \ref{sec:proofs}.

 \section{Energy Conservation Anomalies  in $d=2$ and $d\geq 3$ and Turbulent Cascade Direction}
 \label{sec:anom}

As discussed in the introduction, turbulent fluids are remarkably effective at transferring energy across scales.  In dimensions three and higher, this is reflected by the dissipative anomaly, or persistent dissipation of kinetic energy in the limit of zero-viscosity; in dimension two, the inverse energy cascade can transfer energy which is input by a scale localized force up to large scales even in the limit where the typical forcing wavenumber is taken to infinity.  This is all the more surprising because, in both these cases, the direct effect of viscosity/forcing respectively vanish (at least in the sense of distributions).

Deep insight into the mechanism of such dissipative/input anomaly came from Lars Onsager in his famous 1949 paper \cite{O49}. There, he discussed the idea that weak solutions of the Euler equation may not conserve energy due to a non-linear energy cascade despite the fact that no non-ideal effects are present.   Following these ideas, Duchon \& Robert \cite{DR00} consider any weak solution $\bu$ for the Euler equations with velocity satisfying $\bu\in L^3(0,T; L^3(\mathbb{T}^d))$  and with forcing $f  \in L^2(0,T; L^2(\mathbb{T}^d))$
\bea\label{Eeqn1}
\partial_t \bu+ \nabla \cdot (\bu \otimes \bu) \!\! &=& \!\! -\nabla p+f,\\ 
\nabla \cdot u \!\!&=& \!\!0, \label{Eeqn2}
\eea
 {where have assumed mass-density $\rho_0$ is homogeneous and set to unity.} The transfer of energy through scale can then be described as follows.   Let $G\in C_0^\infty(\mathbb{T}^d)$ be a standard mollifier, $G_\ell(r)=\ell^{-d} G(r/\ell)$ and define\footnote{In their work \cite{DR00}, Duchon \& Robert state this result with a different definition of $\Pi_\ell[u]$.  At finite $\ell$, these two expressions differ.  However, in the limit $\ell\to 0$, both expressions converge to the same limit distribution $\Pi[u]$, see \cite{Enotes}  \S IIIb.} the function
\begin{equation}\label{Dell}
\Pi_\ell[\bu] := -\nabla \bar{\bu}_\ell :\tau_\ell(\bu,\bu)
\end{equation} 
where $\ol{u}_\ell(x) = \int_{\mathbb{T}^d} G_\ell(r) u(x+r) dr$ and $\tau_\ell(u,u)= \ol{(u\otimes u)}_\ell - \ol{u}_\ell\otimes \ol{u}_\ell$.  This term represents energy flux-through-scale and appears as a transfer term in the balance of `resolved' kinetic energy 
\begin{equation}\label{EulerBalanceCG}
\partial_t \left(\frac{1}{2}|\ol{\bu}_\ell|^2 \right) +\nabla \cdot \left[\left(\frac{1}{2}|\ol{\bu}_\ell|^2 +\ol{p}_\ell\right)\ol{\bu}_\ell  +\ol{\bu}_\ell \cdot \tau_\ell(u,u)\right]= -\Pi_\ell[\bu] +\ol{\bu}_\ell\cdot \ol{f}_\ell.
\end{equation}
 By Proposition 2 of \cite{DR00}, as $\ell \to 0$ the functions $\Pi_\ell[\bu]\in L^1((0,T)\times \mathbb{T}^d)$ converge in the sense of distributions on $(0,T)\times \mathbb{T}^d$ to a distribution $\Pi[\bu]$ independent of the mollifying sequence, i.e. $\Pi[u]=\Dlim_{\ell\to 0}\Pi_\ell[u]$ where $\Dlim_{\ell\to 0}$ represents the limit is taken in the sense of distributions.
 Further, taking the $\ell\to 0$ limit of Eq. \eqref{EulerBalanceCG}, one finds that $u\in  L^3(0,T; L^3(\mathbb{T}^d))$ satisfies a local energy balance  { (in the sense of space-time distributions)} which includes a possible anomaly due to singularities in the solution
\begin{equation}\label{EulerBalanceWeak}
\partial_t \left(\frac{1}{2}|\bu|^2 \right) +\nabla \cdot \left[\left(\frac{1}{2}|\bu|^2 +p\right)\bu\right]= -\Pi[\bu] +\bu\cdot f, \qquad  \Pi[u] = \Dlim_{\ell\to 0} \Pi_\ell[\bu].
\end{equation}
The limit $\Pi[u]$ need not vanish due to nonlinear energy cascade facilitated by rough velocity fields.  In fact, Onsager famously conjectured \cite{O49} that, in order to dissipate energy, an Euler solution cannot possess H\"{o}lder regularity $u\in C^\alpha$ with $\alpha>1/3$.  Otherwise $\Pi[\bu]=0$.  Eyink \cite{GLE94} proved this assertion under a slightly stronger assumption and Constantin, E \& Titi \cite{CET94} then proved the sharper result for $u\in L^p(0,T;B_{p}^{\alpha,\infty}(\mathbb{T}^d))$ for any $p\geq3$.   Moreover, weak Euler solutions with $\Pi[u]\neq 0$ are known to exist (see, e.g. \cite{LS12,Shn197}) and recent constructions have demonstrated that the regularity threshold proposed by Onsager is sharp \cite{I16,BLSV17}. 

The work of \cite{DR00}, reviewed in \S \ref{sec:3danom}  below, connects the Euler anomaly $\Pi[\bu]$ to its physical origin in $3d$: the energy dissipation anomaly \eqref{zerothLaw3d} for limits of Navier-Stokes solutions.  In \S \ref{sec:2danom} we extend these considerations to a framework designed to describe an `ideal' inverse cascade in $2d$.  In this setting, we show that  $\Pi[\bu]$  is connected to the anomalous energy input by a force acting only at infinitesimally small scales.

\subsection{Dissipation Anomaly in Dimensions $d\geq 3$ and Direct Cascade} \label{sec:3danom}
The forced Navier-Stokes equations governing the evolution of a viscous incompressible fluid are
\bea\label{NSE}
\partial_t \bu^\nu + \nabla \cdot (\bu^\nu\otimes \bu^\nu) \!\! &=& \!\! -\nabla p^\nu + \nu \Delta \bu^\nu+ f^\nu,\\ \label{incom}
\nabla \cdot u^\nu \!\!&=& \!\!0,
\eea
with solenoidal initial conditions 
$u^\nu|_{t=0}=u_0^\nu\in  L^2({\mathbb T}^d)$ and forcing $f  \in L^2(0,T; L^2(\mathbb{T}^d))$.   If \eqref{NSE} and \eqref{incom} are understood in the sense of distributions on $\mathbb{T}^d\times [0,T]$, then weak solutions in the space $\bu^\nu\in  L^\infty(0,T;L^2({\mathbb T}^d))\cap L^2(0,T; H^1({\mathbb T}^d))$ for $\nu>0$, known as Leray solutions, exist globally but are not known to be unique.
Such solutions satisfy a local (generalized) energy equality \cite{DR00}, which states 
\begin{equation}\label{NSbalanceWeak}
\partial_t \left(\frac{1}{2}|\bu^\nu|^2 \right) +\nabla \cdot \left[\left(\frac{1}{2}|\bu^\nu|^2 +p^\nu\right)\bu^\nu-\nu \nabla \frac{1}{2}|\bu^\nu|^2 \right]=-\varepsilon[\bu^\nu]
\end{equation}
where the energy dissipation rate is
\be\label{epsDef}
\varepsilon[\bu^\nu] := \nu |\nabla \bu^\nu|^2+ D[\bu^\nu]
\ee
 with $D[\bu^\nu]$ a Radon measure that represents dissipation due to possible Leray singularities.   
 
 Freely-decaying and externally-forced incompressible turbulence appear substantially similar for dimensions $d\geq 3$; there is a direct (or forward) cascade of energy from large to small scales.  Moreover for $d=3$, as discussed in the introduction, it is a well known experimental observation that at large Reynolds numbers the dissipation rate becomes independent of $\nu$ and is non-vanishing.  Anomalous dissipation, or the zeroth `law' of turbulence \eqref{zerothLaw3d}, would be reflected mathematically by the property that 
\begin{equation}\label{energyAnom}
 \lim_{\nu\to 0} \varepsilon[\bu^\nu] >0
\end{equation}
 as a distribution, i.e. for some positive test function $\varphi$, $\langle \varepsilon[\bu],\varphi\rangle >0$.  Although the property \eqref{energyAnom} remains a mathematical conjecture for solutions of Navier-Stokes equations, there is a wealth of experimental \cite{KRS84,PKW02} and numerical \cite{KRS98,KIYIU03} evidence that supports it.  See also Remark 3 of \cite{DE17}.

Duchon \& Robert connect (\S3 of \cite{DR00}) the anomalous dissipation \eqref{energyAnom} to properties of weak Euler solutions under the assumption that $u^\nu \to u$ strongly in $L^3(0,T;L^3(\mathbb{T}^d))$.  In particular, they showed that the limit $u\in  L^3(0,T; L^3(\mathbb{T}^d))$ is a weak solution to the incompressible Euler equations 
\eqref{Eeqn1}--\eqref{Eeqn2}
 which additionally satisfies a distribution local energy balance arising as the limit of Eq. \eqref{NSbalanceWeak}
 \begin{equation}\label{NSLimitBalanceWeak}
\partial_t \left(\frac{1}{2}|\bu|^2 \right) +\nabla \cdot \left[\left(\frac{1}{2}|\bu|^2 +p\right)\bu\right]= -\varepsilon[\bu] +\bu\cdot f, \qquad\quad    \varepsilon[u]= \Dlim_{\nu\to 0} \varepsilon[\bu^\nu] .
\end{equation}
Moreover, comparing \eqref{NSLimitBalanceWeak} to the balance equation \eqref{EulerBalanceWeak} valid for general weak Euler solutions $u\in L^3$ space-time, the limiting dissipation matches on to the to the non-linear flux $\Pi[u]$, namely
 \be\label{PiEps}
 \Pi[u] =  \varepsilon[u] 
 \ee
 in the sense of distributions,    {i.e.  it holds when averaged over the same (arbitrary) bounded.  Thus, the identification \eqref{PiEps} is space-time local but may not hold pointwise in $(x,t)\in\mathbb{T}^d\times[0,T]$. 
spacetime region.  Moreover, even if there is a global dissipation anomaly  \eqref{zerothLaw3d}, it may well be that the dissipation is not taking place everywhere and for maybe observation regions (test functions $\varphi$), both the distributions $\Pi[u]$ and $\varepsilon[u]$ appearing in \eqref{PiEps} vanish. }   

The physical picture suggested by \eqref{PiEps} is one of cascade; the energy in the system, possibly input at large-scales $L$ by an external force, cascades downscale $\ell\lesssim L$ through nonlinear transfer $\Pi_\ell[u]$ until the smallest scales where it is dissipated by the action of viscosity.  At those smallest scale, the nonlinear flux and viscous dissipation balance.  
This is often termed a direct cascade, which is reflected by the fact the flux is asymptotically downscale $ \Pi[u]\geq0$ by the identification \eqref{PiEps}\footnote{ {We remark that this is an asymptotic statement related to the cascade at arbitrarily small-scales.  It does not imply that the cascade rate is constant (or even positive) throughout all scales in the inertial range, although in practice this is very often observed.}}.  This corroborates Onsager's picture of  infinite--$Re$ number turbulence as being governed by dissipative weak Euler solutions, as it directly relates anomalous dissipation as $\nu\to 0$ to the inviscid limit of viscous energy dissipation of Navier-Stokes solutions. 

  {It is worthwhile mentioning that there is a strong connection between the formula \eqref{PiEps} and the celebrated Kolmogorov $4/5$--law.  To be precise, for any weak Euler solution $u\in L^3(0,T;L^3(\mathbb{T}^d))$, consider the \emph{longitudinal third-order structure function} $ S^L_r[u] \in L^1([0,T]\times \mathbb{T}^d)$ defined by
 \be
 S^L_r[u] := \frac{1}{|r|} \int_{S^{d-1}} \delta u_L(r)^3  \ \rmd \omega(\hat{r}) 
 \ee
 where $\rmd \omega(\hat{r}) $ is the unit Haar measure on $S^{d-1}$ and $\delta u_L(r;x,t) := \hat{r}\cdot \delta u (r;x,t)$ is the longitudinal velocity increment.  In Corollary 1 of \cite{E02},  Eyink  proved that if the distributional limit $ S^L[u]:= \Dlim_{r\to 0} S^L_r[u]$ exists, then it matches onto the the limit of the non-linear flux, $\Pi[u]$.   Specifically, he established the equality \be\label{45thlaw}
 S^L[u] = -\frac{12}{d(d+2)}\Pi[u] ,
 \ee 
 interpreted in the sense of distributions on $[0,T]\times \mathbb{T}^d$.  It follows from Eq. \eqref{45thlaw} and  the identification \eqref{PiEps}  that for any strong limit $u^\nu \to u$  in $L^3(0,T;L^3(\mathbb{T}^d))$ one has $ S^L[u] = -\frac{12}{d(d+2)}\varepsilon[u] $, recovering the usual $4/5$--law in a space-time local sense (no ensemble averaging necessary) in three dimensions.  With this identification in hand, together with \eqref{PiEps} and \eqref{PiForce}, \eqref{45thlaw} and be used to replace $\Pi[u]$, $\varepsilon[u]$ and $-I[u]$ with $ S^L[u]$ in the formulae \eqref{Eulerasym}, \eqref{3dasym} and \eqref{2dasym} respectively, so that these expressions may indeed be regarded as ``Lagrangian analogues" of the $4/5$--law.
  }

\subsection{Anomalous Input in Dimension $d= 2$ and Inverse Cascade}\label{sec:2danom}

 Kraichnan, in a seminar paper \cite{Kr67}, argued that, in the limit of small viscosity, most of the energy input by forcing would cascade to larger scales because of the ``spectral blocking" effect \cite{Fjor53,E96} of the enstrophy flux, with only very little energy `leaking' to small scales. Using dimensional reasoning and physical arguments, Kraichnan proposed that a dual cascade should occur, i.e. there should be a  \emph{inverse energy cascade range} at scales greater than the typical forcing scale $\ell_f$, and also  a \emph{direct enstrophy cascade range} at scales smaller than $\ell_f$.  Moreover, he predicted that the energy spectrum $E(k)$ scales in these ranges as
\be\label{2dSpec}
E(k)\sim \begin{cases} I^{2/3} k^{-5/3},&   k \ll k_f \\
\eta^{2/3} k^{-3},&   k \gg k_f
\end{cases},
\ee
where $I$ is the energy injection rate by forcing, $\eta$ is the enstrophy injection rate\footnote{More correctly, the dual cascade picture was predicted by Kraichnan to occur in a statistically steady state for a fluid with large-scale damping (such as linear friction of hyperviscosity) and viscosity.  These two effects  impost cutoff wavenumbers; damping imposes $k_{ir}$ is an infrared cutoff and viscosity $k_{uv}$ is the corresponding ultraviolet.  Then, the inverse energy cascade range is predicted to be confined to $ k_{ir} \ll k \ll k_f$ whereas the direct enstrophy range to $k_{f} \ll k \ll k_{uv}$.  For simplicity, in our analysis, we consider forced Euler equations, neglecting the effects of large-scale damping and viscosity.  However, our conclusions can easily be modified to accommodate the presence of a damping term and for Navier-Stokes solutions in the limit where viscosity $\nu$ is taken to zero before all others discussed in this section.}.  These conclusions were proposed independently by Batchelor \cite{B69} for freely decaying $2d$ turbulence.  There were derived also by Eyink \cite{E96} who provided a more rigorous basis for the theory using somewhat different arguments.

As discussed in the introduction, we are interested in an ideal inverse cascade setup. As a simplified model, we consider the ideal Euler equations with small-scale forcing in the limit where the force acts only at infinitesimally small-scales.   This limit should result in an inverse energy cascade range permeating to all scales.   To make our setup precise, we consider weak solutions to the forced Euler equations  $\bu^{k_f}\in C([0,\infty);W^{1,r}(\mathbb{T}^2))$  with a forcing which is spectrally concentrated at wavenumber $k_f$,
\bea\label{EE}
\partial_t \bu^{k_f} + \nabla \cdot (\bu^{k_f}\otimes \bu^{k_f}) \!\! &=& \!\! -\nabla p^{k_f} + f^{k_f},\\
\nabla \cdot u^{k_f} \!\!&=& \!\!0,
\eea
with initial data $u_0^{k_f}\in L^2(\mathbb{T}^2)$ and $\omega_0^{k_f}\in L^r(\mathbb{T}^2)$ for $r\in (3/2,\infty)$ and forcing $ f^{k_f}\in L^2(0,T;L^2(\mathbb{T}^2))$.  The existence of at least one such weak solution is guaranteed provided only $r\in (1,\infty)$, see \cite{L96}.  Our restriction that $r> 3/2$ ensures, by  Proposition 6 of \cite{DR00}, that $u^{k_f}$ satisfies the distributional energy balance
\begin{equation}\label{EbalanceWeak2d}
\partial_t \left(\frac{1}{2}|\bu^{k_f}|^2 \right) +\nabla \cdot \left[\left(\frac{1}{2}|\bu^{k_f}|^2 +p^{k_f}\right)\bu^{k_f} \right]=I[\bu^{k_f}], \qquad I[\bu^{k_f}]:= u^{k_f}\cdot f^{k_f}
\end{equation}
where the energy input is due to solely to the forcing.  In particular, there is no anomalous term in the energy balance \eqref{EbalanceWeak2d} arising from possible singularities in the solutions (such as, for example, the $D[u^\nu]$ distribution which appeared in \eqref{epsDef}).    We are interested in the limit in which the typical forcing wavenumber is taken  off to infinity $k_f\to \infty$, or equivalently $\ell_f\to 0$.    Such a force will have ``infinite frequency" and be zero from the distributional point of view (a concrete example is presented in Proposition \ref{forcingProp}).   However, analogous to the dissipative anomaly \eqref{energyAnom}, there may be remnant input of energy from the forcing \eqref{zerothLaw2d}, expressed mathematically by
 \begin{equation}\label{forcingAnom}
\lim_{k_f\to \infty} I[\bu^{k_f}] \neq 0
\end{equation} 
in the sense of distributions.
We call the property \eqref{forcingAnom}  ``anomalous input" or a ``production anomaly" since we expect that typically the role of the forcing is to act as a source for energy rather than a sink\footnote{This depends, of course, on the choice of forcing scheme. For example, energy input is ensured if the forcing is chosen to be solution--dependent, e.g. small-scale Lundgren forcing of the form $f=\alpha \bold{P}_{k_f}[\bu]$ with $\alpha:=\alpha(k_f)>0$ and  $\bold{P}_{k_f}$ is the projection onto a shell around $k_f$ in wavenumber space.  Another attractive choice of force is to take $f$ to be a homogenous Gaussian random field which is white-noise correlated in time, i.e. $\langle f_i (\bx,t) f_j(\bx',t')\rangle = 2F_{ij}(\bx-\bx')\delta(t-t')$.   This has the theoretical advantage that, after averaging over the forcing statistics, the mean injection rate of energy is \emph{solution independent}, i.e. after averaging the balance \eqref{EulerBalanceWeak}, the injection term is $\langle \bu\cdot f\rangle=F_{ii}(0)>0$, insuring input of energy on average. }.  It is called anomalous because it is fed into the flow at infinitely small scales, where irregular turbulent motion is required to facilitate energy transfer up through the inertial range and into the largest scales of the flow.  There are examples of flows with input anomalies of the form \eqref{forcingAnom}.  In fact, Shnirelman's \cite{Shn197} original construction of non-conservative weak Euler solutions have input anomaly $I[\bu] \neq 0$ and are motivated by the physical idea of the inverse energy cascade.  

We now specify details on the forcing schemes we consider.  We require that $ f^{k_f}\in L^\infty(0,T;L^2(\mathbb{T}^2))$  for all $k_f<\infty$ and that  $f^{k_f}\to 0$ in the sense of distribution as $k_f\to\infty$.   This can easily be accomplished, for example, by considering a force with compact spectral support and taking the forcing wavenumber $k_f$ off to infinity (or equivalently the typical forcing length scale $\ell_f= 2\pi/k_f$ is taken to zero).  Indeed

\begin{prop}\label{forcingProp}
Let  $f^{k_f}$ have spectral support inside a band $[k_f/2, 2k_f]$ around some wavenumber $k_f\in(0,\infty)$.  Further, assume that $f^{k_f}\in  L^2(0,T; L^2(\mathbb{T}^2))$ for all $k_f<\infty$ and with $L^1$ norms satisfying $\|f^{k_f}\|_{L^1(0,T; L^1(\mathbb{T}^2))}\leq C k_f^N$ for any $N>0$.  Then $f^{k_f}\to 0$ in the sense of distributions as $k_f\to \infty$.
\end{prop}

The proof of the proposition is elementary and is differed to \S \ref{sec:proofs}.  
Note that we do not explicitly specify how the amplitudes of the forcing depend on $k_f$; only that the family of forces have  space-time $L^1$ norms bounded by an arbitrary power of $k_f$.  Indeed, it is important  that the norms $\|f^{k_f}\|_{L^2(0,T; L^2(\mathbb{T}^2))}$ not be uniformly bounded in $k_f$.  Otherwise $\lim_{k_f\to \infty} I[\bu^{k_f}] = 0$ as we demonstrate in Remark \ref{forcingRem} of  \S \ref{sec:proofs}.  Thus, the forcing we consider is simultaneously required to act within bands of increasingly high wavenumbers, and have diverging amplitude.  Such forcing schemes have very little restriction\footnote{We are grateful to P. Isett for pointing out an improvement of Prop \ref{forcingProp} from an early preprint which we present here.}, leaving plenty of room to create an input anomaly of the type \eqref{forcingAnom}.

 It is now straightforward, following the approach of \cite{DR00}, to connected the anomaly \eqref{forcingAnom} to dissipative properties of weak \emph{unforced} Euler solutions under the assumption that $u^{k_f} \to u$ strongly in $L^3(0,T;L^3(\mathbb{T}^2))$ provided that $f^k\to 0$ in the sense of distributions as $k_f\to \infty$ (for example, forcing given by Proposition \ref{forcingProp}).  
Then, it is easy to see that the limit $u\in L^3(0,T;L^3(\mathbb{T}^2))$ is a weak solution to the unforced incompressible Euler equations \eqref{Eeqn1}--\eqref{Eeqn2} with $f\equiv 0$ 
 which satisfies the local energy balance arising as the limit of Eq. \eqref{EbalanceWeak2d}
 \begin{equation}\label{ForcedLimitBalanceWeak}
\partial_t \left(\frac{1}{2}|\bu|^2 \right) +\nabla \cdot \left[\left(\frac{1}{2}|\bu|^2 +p\right)\bu\right]= I[\bu],   \qquad\quad   I[u]= \Dlim_{k_f\to \infty} I[\bu^{k_f}].
\end{equation}
The details of this argument are very similar to those given in  \cite{DR00} and are provided in Chapter 3 of \cite{DE17Thesis}. 
Comparing \eqref{NSLimitBalanceWeak} to the balance equation \eqref{EulerBalanceWeak}, which is valid for general weak Euler solutions $u\in L^3$ space-time, the  limiting energy input matches on to the to the negative flux $-\Pi[u]$, namely
 \be\label{PiForce}
 \Pi[u] = -I[u] 
 \ee
 in the sense of distributions.    Again, as we expect the forcing to input energy into the flow, the equality \eqref{PiForce} implies that $ \Pi[u]<0$ for ``typical" forcing schemes.   The physical picture is that, despite the fact that there is no direct forcing in the momentum equation \eqref{Eeqn1},  energy is  fed into the system by the ``infinite" frequency forcing acting at ``infinitesimally" small scales, where it is transferred upscale via a nonlinear inverse cascade until it accumulates at large-scales.

 \section{Proofs}\label{sec:proofs}

 \begin{proof}[Proof of Theorem \ref{theorem}]
Parts (i), (ii) and (iii) have the same proof up to the application of Lemma \ref{lemma}.  Throughout the proof, all super-scripts indicating parametric dependence on $\nu$ or $k_f$ are omitted.

The mollified velocity $\ol{u}_\ell(t,\cdot)$ is $C^\infty(\mathbb{T}^d)$ as a function of space for every time $t\in[0,T]$, and for all $x\in\mathbb{T}^d$, and  {the function $t\mapsto \ol{u}_\ell(x,t)$  is Lipschitz continuous uniformly in $x$. This time-regularity of the mollified field is inherited from the equations of motion, as we now show in the following proposition.}

\begin{prop}\label{timeReg}
 {Let $u\in L^\infty(0,T;L^2(\mathbb{T}^d))$ be any weak solution of the incompressible Navier-Stokes or Euler equations with forcing $f\in L^\infty(0,T;L^2(\mathbb{T}^d))$.  Then the mollified velocity $\ol{u}_\ell(\cdot,x)$ is Lipschitz in time, uniformly in space.}
\end{prop}

\begin{proof}[Proof of Proposition \ref{timeReg}]
 { We work with Navier-Stokes solutions, Euler solutions follow by the same argument.  Choosing test functions of the form $\varphi(t, \cdot) := \psi(t) G_\ell(x-\cdot)$, we see that any weak solutions of Navier--Stokes satisfy the mollified equations pointwise for $x\in \mathbb{T}^d$ and distributionally 
for $t\in [0,T]$:
\be\label{weak}
\partial_t \ol{\bu}_{\ell} + \nabla \cdot \ol{(\bu \otimes \bu)}_{\ell} = - \nabla \ol{p}_{\ell}  
+ \nu \Delta  \ol{\bu}_{\ell}+ \ol{f}_{\ell}.
\ee
We aim to establish a uniform-in-$x$ bound for $\partial_t \ol{\bu}_{\ell}(x,\cdot)$  in $L^\infty([0,T])$. Since $u \in L^\infty([0,T]; L^2(\mathbb{T}^d))$,  then for every $x\in \mathbb{T}^d$ we have by Young's convolution inequality that
\begin{align}
 \|\nabla \cdot \ol{(\bu \otimes \bu)}_{\ell}(x,\cdot)\|_{L^\infty([0,T])}
&\leq \frac{1}{\ell}\|(\nabla G)_\ell\|_\infty \|u\|^2_{L^\infty([0,T]; L^2(\mathbb{T}^d))},\\ 
 \|\nu \Delta  \ol{\bu}_{\ell}(x,\cdot)\|_{L^\infty([0,T])}
&\leq \frac{\nu}{\ell^2}\|(\Delta G)_\ell\|_2 \|u\|_{L^\infty([0,T]; L^2(\mathbb{T}^d))},\\
 \| \ol{f}_\ell(x,\cdot)\|_{L^\infty([0,T])} &\leq \| G_\ell\|_2 \|f\|_{L^\infty([0,T]; L^2(\mathbb{T}^d))}.
\label{point-est1} \end{align}
 The pressure-gradient term $\nabla \ol{p}_{\ell}(x,t)$ in (\ref{weak}) is determined using $\nabla\cdot f=0$ from the Poisson 
equation 
\be -\Delta \nabla \ol{p}_{\ell}(\cdot,t)= (\nabla\otimes\nabla\otimes\nabla):\ol{(\bu \otimes \bu)}_{\ell}(\cdot,t).
\ee
Note that the righthand-side belongs to $C^\infty(\mathbb{T}^d)$ for a.e. time $t$.  The solution of the Poisson problem therefore satisfies the following Sobolev estimate for any integer $m> 2$
\be \label{pressbnd}
\|\nabla \ol{p}_{\ell}(\cdot,t)\|_{H^m(\mathbb{T}^d)}
\leq C \|  (\nabla\otimes\nabla\otimes\nabla):\ol{(\bu \otimes \bu)}_{\ell}(\cdot,t) \|_{H^{m-2}(\mathbb{T}^d)}
\ee 
for some constant $C$.  The righthand-side above can be bounded as follows
\begin{align}\nonumber
  \|  (\nabla\otimes\nabla\otimes\nabla):\ol{(\bu \otimes \bu)}_{\ell}(\cdot,t) \|_{H^{m-2}(\mathbb{T}^d)}&\leq C\| \nabla^{(m+1)} \ol{(\bu \otimes \bu)}_{\ell}(\cdot,t) \|_{L^\infty(\mathbb{T}^d)}\\ \label{pressnonlinebdd}
  & \leq C \|\nabla^{(m+1)}G\|_\infty \|u(\cdot,t)\|^2_{L^2(\mathbb{T}^d)}/\ell^{m+1+d}
\end{align}
for some constant $C$. Choosing $m>d/2$, by virtue Sobolev embedding $H^m(\mathbb{T}^d) \hookrightarrow L^\infty(\mathbb{T}^d)$ we have from \eqref{pressbnd} and \eqref{pressnonlinebdd} that 
\be 
\|\nabla \ol{p}_{\ell}(x,\cdot)\|_{L^\infty([0,T])} \leq  C \|\nabla^{(m+1)}G\|_\infty \|u\|^2_{L^\infty([0,T];L^2(\mathbb{T}^d))}/\ell^{m+1+d},
\ee 
 for some constant $C$
and every $x\in \mathbb{T}^d$.
We thus see that every term in (\ref{weak}) for the distributional 
derivative $\partial_t \ol{\bu}_{\ell}(x,\cdot)$ belongs to $L^\infty([0,T])$ so that 
$\ol{\bu}_{\ell}(x,\cdot)$ for every $x\in\mathbb{T}^d$ is Lipschitz continuous. }
\end{proof}

 {
Note moreover that, since $\bu\in L^\infty(0,T;L^2(\mathbb{T}^d))$ (for parts (ii) and (iii), with norms uniformly bounded in $\nu$ and $k_f$), the mollified field and all its derivatives are uniformly bounded in $\bx$ at fixed $\ell>0$ 
\be
\|\nabla^{(n)}\ol{\bu}_\ell\|_{L^\infty([0,T]\times \Omega)} \leq  \|\nabla^{(n)}G\|_\infty \|\bu\|_{L^\infty(0,T;L^1(\mathbb{T}^d))}  /\ell^{n+d}.
\ee
 Therefore $\ol{u}_\ell\in {\rm Lip}([0,T] \times \mathbb{T}^d)$ (it is actually much more regular in space;  $C_x^\infty$ for a.e. $t\in[0,T]$),  the Lagrangian particle trajectories defined by Eqn. \eqref{eqn:traj} exist and are unique.  
It follows from the equation $\dot{\bX}^{\ell}_{t_0,t}= \ol{\bu}_\ell (\bX_{t_0,t}^{\ell},t)$ that $\bX_{t_0,t}^{\ell}(x)\in C^2([0,T])$ uniformly in $\bx$ for fixed $\ell>0$ since 
\be\label{ddotreg}
\ddot{\bX}^{\ell}_{t_0,t}(x) = a^{\ell} (\bX_{t_0,t}^{\ell}(x),t)\in L^\infty([0,T])
\ee
where we have introduced $a^{\ell}(x,t)$, the material derivative of the mollified velocity or the large-scale Eulerian acceleration 
\be\label{filtAcc}
a^{\ell} (x,t)= \left(\partial_t \ol{u}_\ell+ \ol{u}_\ell\cdot \nabla \ol{u}_\ell\right)(x,t).
\ee
The claimed regularity \eqref{ddotreg} follows from the fact that $\partial_t \ol{\bu}_\ell (\cdot, x)\in  L^\infty([0,T])$ and the bound
\be\label{accelbnd}
\| a^{\ell} (x,\cdot)\|_{L^\infty([0,T])}\leq \|\partial_t \ol{\bu}_\ell (\cdot, x)\|_{L^\infty([0,T])}+ \frac{1}{\ell} \|G_\ell\|_{\infty} \|(\nabla G)_\ell\|_{\infty}\|u\|^2_{L^\infty([0,T]; L^2(\mathbb{T}^d))}.
\ee
Thus, since $\bX_{t_0,t}^{\ell}(x)\in C^2([0,T])$ for each $x\in \mathbb{T}^d$, it follows by Taylor's theorem that for any $\bx\in \mathbb{T}^d$ there exist functions $h_f:=h_f(\tau;t,x,r,\ell)$ and $h_b:=h_b(\tau;t,x,r,\ell)$ with the properties that $\lim_{\tau\to 0}h_f(\tau)=\lim_{\tau\to 0}h_b(\tau)=0$ and are such that the following short-time expansion for trajectories both forwards and backwards in time hold}
\begin{align}\label{fwdTraj}
\delta \bX_{t,t+\tau }^{\ell}(r;x) - r &= \phantom{-}\delta \ol{u}_\ell(r;x,t)\tau    + {\frac{1}{2}}\delta a^{\ell} (r;x,t) \tau^2 + {h_f(\tau)}\tau^2,\\
\delta \bX_{t,t-\tau }^{\ell}(r;x) - r &= -\delta \ol{u}_\ell(r;x,t)\tau    +  {\frac{1}{2}}\delta a^{\ell} (r;x,t) \tau^2 +   {h_b(\tau)} \tau^2,\label{bwdTraj}
\end{align}
 { Note that, for fixed $\ell>0$, the fields $\delta \ol{u}_\ell(r;x,t),\delta a^{\ell} (r;x,t),\delta \bX_{t,t+\tau }^{\ell}(r;x) \in L^\infty([0,T])$ uniformly in $x, r\in \mathbb{T}^d$ by the fact that $u\in L^\infty(0,T;L^2(\mathbb{T}^d))$ and  the bound \eqref{accelbnd}.  Moreover, for fixed $\ell>0$, all these fields are $C^\infty$ in the variables $x$ and $r$ for a.e. $t\in [0,T]$.   Since equations \eqref{fwdTraj} and \eqref{bwdTraj} in fact define $h_f$ and $h_b$, it follows that these are smooth functions in the variables $x$ and $r$ for a.e. $t$ and bounded in time for all $x,r\in \mathbb{T}^d$.   Using these facts, squaring \eqref{fwdTraj}, \eqref{bwdTraj} and integrating in $r$ against $\psi_R$, we obtain the expansions for the relative dispersion up to $o(\tau^3)$ errors\footnote{Where the notation $f(\tau)=o(\tau^3)$ denotes $\lim_{\tau\to 0} f(\tau)/\tau^3\to 0$.}}
\begin{align} \label{reldispF}
\Big\langle | \delta \bX_{t,t+\tau}^{\ell}(\br;\bx)-&\br|^2\Big\rangle_R =   \langle S_2^{\ol{\bu}_\ell}(\br,t)\rangle_R\  \tau^2  +   \frac{1}{2} \frac{\rmd}{\rmd \tau} \langle |\delta_r {v}^{\ell}(\tau;x,t)|^2\rangle_R\Big|_{\tau=0} \tau^3 +  {o(\tau^3)},\\ \label{reldispB}
\Big\langle | \delta \bX_{t,t-\tau}^{\ell}(\br;\bx)-&\br|^2\Big\rangle_R =   \langle S_2^{\ol{\bu}_\ell}(\br,t)\rangle_R\  \tau^2  -   \frac{1}{2} \frac{\rmd}{\rmd \tau} \langle |\delta_r {v}^{\ell}(\tau;x,t)|^2\rangle_R\Big|_{\tau=0} \tau^3 +  {o(\tau^3)},
\end{align} 
where we defined $v^{\ell}(\tau,x;t):= \ol{u}_\ell (X_{t,t+\tau}^{\ell}(x),t+\tau)$ and  $\delta_r  \bv^{\ell}(\tau;\bx,t):=  v^{\ell}(\tau,x+r;t)-v^{\ell}(\tau,x;t)$. 
In writing \eqref{reldispF}, \eqref{reldispB}, we used the notation $S_2^{\ol{\bu}_\ell}(\br,t):= |\delta \ol{u}_\ell(r;x,t)|^2$ and the fact that
\begin{align}\label{EulRepRelVelo}
\frac{1}{2} \frac{\rmd}{\rmd \tau} \langle |\delta_r {v}^{\ell}(\tau;x,t)|^2\rangle_R\Big|_{\tau=0}  = \langle \delta \ol{\bu}_\ell(r;x,t) \cdot \delta {a}^{\ell}(r;x,t)\rangle_R.
\end{align}
Subtracting the forward dispersion from the backward, dividing by $\tau^3$ and taking the limit $\tau\to 0$, we have
\be\label{diffBal}
\frac{ \Big\langle | \delta \bX_{t,t+\tau}^{\ell}(\br;\bx)-\br|^2\Big\rangle_R -\Big\langle | \delta \bX_{t,t-\tau}^{\ell}(\br;\bx)-\br|^2\Big\rangle_R }{2\tau^3}\ \xrightarrow{\tau\to 0} \  \frac{1}{2} \frac{\rmd}{\rmd \tau} \langle |\delta_r {v}^{\ell}(\tau;x,t)|^2\rangle_R\Big|_{\tau=0} .
\ee
Next taking the limit $\nu\to 0$ in $d\geq3$ and $k_f\to\infty$ in $d=2$, and finally, taking $R,\ell\to 0$ (in the sense of distributions in $x,t$) and applying the Lemma \ref{lemma}, we obtain the formulae \eqref{Eulerasym}, \eqref{3dasym} and \eqref{2dasym}.
\end{proof}

 \begin{proof}[Proof of Lemma \ref{lemma}]
Fix any $\phi\in C_0^\infty([0,T]\times\mathbb{T}^d)$, consider $\varphi(x,r,t) = \phi(x,t)\psi_R(r)$ and denote 
\be
 \langle f(x,r,t) \rangle_\varphi:=\int_0^T \int_{\mathbb{T}^d\times \mathbb{T}^d} f(x,r,t) \varphi(x,r,t) \ \rmd t \rmd x \rmd r.
\ee

\noindent \emph{Proof of Lemma \ref{lemma} (i)}
The Eulerian acceleration increment \eqref{filtAcc} from the mollified Euler equations is
\begin{align}
 \delta {a}^{\ell}(r;x) \equiv -\nabla_x \delta\ol{p}_\ell(r;x) + \delta \ol{f }_\ell(r;x) -\nabla_x\cdot \delta \tau_\ell (r;x) 
\end{align}
where $\delta\tau_\ell(r;x)  = \tau_\ell(\bu,\bu)(\bx+r)- \tau_\ell(\bu,\bu)(\bx)$ with $\tau_\ell(f,g):= \ol{(fg)}_\ell-\ol{f}_\ell \ol{g}_\ell$ for any $f,g\in L^2(\mathbb{T}^d)$.
 Thus, we have from \eqref{EulRepRelVelo} that
\begin{align} \nonumber
\frac{1}{2} \frac{\rmd}{\rmd \tau} \langle |\delta_r {v}^{\ell}(\tau;x,t)|^2\rangle_{\vphi}\Big|_{\tau=0}  =& -\langle \delta  \ol{u}_\ell(\br;\bx) \cdot\nabla_x   \delta\ol{p}_\ell(r;x) \rangle_\vphi +\langle \delta \ol{ u}_\ell(\br;\bx) \cdot \delta \ol{f}_\ell(\br;\bx) \rangle_\vphi \\
  &- \langle  \delta \ol{u}_\ell(\br;\bx) \cdot\nabla_x\cdot \delta \tau_\ell (\br;\bx) \rangle_\vphi.
\end{align}
We estimate each of these contributions separately.  First we treat the pressure-work term. Since we are assuming $u\in L^{3}(0,T;L^{3}(\mathbb{T}^d))$, we have that $u\otimes u \in  L^{3/2}(0,T;L^{3/2}(\mathbb{T}^d))$  and therefore by strong continuity of Calderon-Zygmund operators in $L^p$ for $1<p<\infty$, it follows that $p\in L^{3/2}(0,T;L^{3/2}(\mathbb{T}^d))$.  Then, by incompressibility, 
\begin{align}\nonumber
\langle \delta  \ol{u}_\ell(\br;\bx) \cdot&\nabla_x   \delta\ol{p}_\ell(r;x) \rangle_\vphi = \langle \nabla_x \cdot[ \delta  \ol{u}_\ell(\br;\bx)   \delta\ol{p}_\ell(r;x)  ]\rangle_\vphi\\
=&-  \int_0^T\int_{\mathbb{T}^d\times{\rm supp}( \psi_R)}\psi_R(r) \nabla\phi(\bx,t) \cdot \delta  \ol{u}_\ell(\br;\bx,t) \delta\ol{p}_\ell(r;x,t)  \rmd t \rmd x dr.
\end{align}
By H\"{o}lder's inequality and the fact that $ {\rm supp}( \psi_R)\subseteq B_R(0)$, we have that
\begin{align}\nonumber
| \langle \delta  \ol{u}_\ell(\br;\bx) \cdot\nabla_x   \delta\ol{p}_\ell(r;x) \rangle_\vphi | &\leq \  \|\psi\|_1 \|\nabla\phi\|_{\infty} \sup_{|\br|<R}\| \delta  \ol{u}_\ell(\br;\cdot)  \|_3 \sup_{|\br|<R}\|  \delta\ol{p}_\ell(r;\cdot)\|_{3/2}\\
 &\leq \  \|G\|_1^2\|\psi\|_1 \|\nabla\phi\|_{\infty} \sup_{|\br|<R}\| \delta  {u}(\br;\cdot)  \|_3 \sup_{|\br|<R}\|  \delta{p}(r;\cdot)\|_{3/2}
\end{align}
where we used Young's inequality for convolutions to remove the mollification.  Thus, we obtain an upper bound independent of $\ell$ which vanishes as $R\to 0$ by strong continuity of shifts in $L^p$ for $1<p<\infty$
\begin{align}
| \langle \delta  \ol{u}_\ell(\br;\bx) \cdot\nabla_x   \delta\ol{p}_\ell(r;x) \rangle_\vphi |  \xrightarrow{R,\ell \to 0}\ 0.
\end{align}
  Similar arguments show that
\begin{align}
 |\langle \delta  \ol{u}_\ell(\br;\bx) \cdot \delta \ol{f}_\ell(\br;\bx) \rangle_\vphi| \ \leq \   \|G\|_1^2 \|\psi\|_1 \|\phi\|_{\infty}\sup_{|\br|\leq R} \|\delta {\bu}(\br;\cdot) \|_2\sup_{|\br|\leq R}\| \delta {f} (\br;\cdot)\|_2 \xrightarrow{R,\ell\to 0}\  0.
\end{align}
 Finally, we estimate the contribution of the turbulent flux:
\begin{align}\nonumber
\langle \delta  \ol{u}_\ell(\br;\bx) &\cdot\nabla_x\cdot \delta \tau_\ell (\br;\bx) \rangle_\vphi =  -\langle \nabla_x\delta   \ol{u}_\ell(\br;\bx) : \delta \tau_\ell (\br;\bx) \rangle_\vphi \\
&\qquad - \int_0^T \int_{\mathbb{T}^d\times{\rm supp}( \psi)}\psi_R(r) \nabla\phi(\bx,t)  \otimes  \delta   \ol{u}_\ell(\br;\bx,t) : \delta \tau_\ell (\br;\bx,t)\rmd t \rmd x \rmd r. \label{fluxTerms}
\end{align}
The second term is easily seen to vanish as $\ell\to 0$ at fixed $R$ since
 \begin{align}\nonumber
&\left|\int_0^T \int_{\mathbb{T}^d\times{\rm supp}( \psi)}\psi_R(r) \nabla\phi(\bx,t)  \otimes  \delta   \ol{u}_\ell(\br;\bx,t) : \delta \tau_\ell (\br;\bx,t)\rmd t \rmd x \rmd r\right| \\ \nonumber
& \hspace{40mm}  \leq  \  \|\nabla_x\phi\|_{\infty}  \sup_{|\br|<R}\| \delta   \ol{u}_\ell(\br;\cdot) \|_3 \sup_{|\br|<R}\|\delta  \tau_\ell (\br;\cdot) \|_{3/2}\\
& \hspace{40mm}  \leq  \ 2 \|G\|_1 \|\nabla_x\phi\|_{\infty}  \sup_{|\br|<R}\| \delta   u(\br;\cdot) \|_3 \sup_{|\br|<R}\|  \tau_\ell(u,u) \|_{3/2}.
\end{align}
We now use the  $L^{p}$ commutator estimate for the coarse-graining cumulant, which states
\begin{equation}\label{CGcumbound}
\| \tau_\ell (f,g) \|_{p}\lesssim   \sup_{|\br|<\ell} \|\delta {f}(\br;\cdot) \|_{2p} \sup_{|\br|<\ell}  \|\delta {g}(\br;\cdot) \|_{2p} .
\end{equation}
See e.g. \cite{CET94} or, more generally, Proposition 3 of \cite{DE17comp}.  Returning to our estimate, we have
 \begin{align}\nonumber
 &\left|\int_0^T \int_{\mathbb{T}^d\times{\rm supp}( \psi)}\psi_R(r) \nabla\phi(\bx,t)  \otimes  \delta   \ol{u}_\ell(\br;\bx,t) : \delta \tau_\ell (\br;\bx,t)\rmd t \rmd x \rmd r\right| \\
& \hspace{60mm} \leq 4 \|G\|_1  \  \|\nabla_x\phi\|_{\infty}\|u\|_3 \sup_{|\br|<\ell}\|\delta {\bu}(\br;\cdot) \|_3^2 \xrightarrow{\ell \to 0} \  0,
\end{align}
which follows from strong continuity of shifts in $L^3$.  

 The remaining terms on the left-hand side of  Eq.  \eqref{fluxTerms} may be expressed as
 \begin{align} \nonumber
\langle \nabla_x\delta   \ol{u}_\ell(\br;\bx) : \delta \tau_\ell (\br;\bx) \rangle_\vphi &=  \langle \nabla_x  \ol{u}_\ell(\bx+\br) :  \tau_\ell (\bx+\br) \rangle_\vphi+ \langle \nabla_x  \ol{u}_\ell(\bx) :  \tau_\ell (\bx)\rangle_\vphi \\
& \qquad -\langle \nabla_x  \ol{u}_\ell(\bx+\br) :  \tau_\ell (\bx)\rangle_\vphi -\langle \nabla_x  \ol{u}_\ell(\bx) :  \tau_\ell (\bx+\br)\rangle_\vphi.
\end{align}
Note that $\nabla_x  \ol{u}_\ell(\bx+\br)=\nabla_r  \ol{u}_\ell(\bx+\br) = \nabla_r  \ol{u}_\ell(\br;x)$ since the role of $x$ and $r$ is symmetric and $\nabla_r  \ol{u}_\ell(x)$. The final term can also be written in a similar form.  After changing variables, it becomes
 \begin{align*}
\int_0^T\!\!\! \int_{\mathbb{T}^d\times\mathbb{T}^d} \vphi(x,r,t)  \nabla_x  \ol{u}_\ell(\bx) :  \tau_\ell (\bx+\br) \rmd t \rmd x dr&= -\int_0^T\!\!\!   \int_{\mathbb{T}^d\times\mathbb{T}^d} \vphi(x-r,r,t)  \nabla_r  \ol{u}_\ell(x-r,t) :  \tau_\ell (x) \rmd t \rmd x dr\\
 & \!\! \! \!  = -\int_0^T \!\!\!  \int_{\mathbb{T}^d\times\mathbb{T}^d} \vphi(x-r,r,t)  \nabla_r \delta \ol{u}_\ell(-r;x,t) :  \tau_\ell (x) \rmd t \rmd x dr.
\end{align*}
Finally, note that the first two terms can be written as $\langle \Pi_\ell[\bu](\bx+\br)\rangle_\vphi$ and $\langle\Pi_\ell[\bu](\bx)\rangle_\vphi$ respectively using the definition resolved energy flux term given by \eqref{Dell}. Therefore, after changing variables in the final term, 
 \begin{align}\nonumber
\langle \nabla_x&\delta   \ol{u}_\ell(\br;\bx) : \delta \tau_\ell (\br;\bx) \rangle_\vphi \\\nonumber
& =\langle \Pi_\ell[\bu](\bx+\br)\rangle_\vphi+ \langle\Pi_\ell[\bu](\bx)\rangle_\vphi\\\label{line2}
&\ \ \ +  \int_0^T \int_{\mathbb{T}^d\times\mathbb{T}^d} \nabla_r \varphi(\bx,\br,t) \cdot \delta \ol{u}_\ell (\br;\bx,t) \cdot \tau_\ell (\bx,t) \rmd t \rmd xdr\\\label{line2.5}
&\ \ \  - \int_0^T \int_{\mathbb{T}^d\times\mathbb{T}^d}  \nabla_r [\varphi(\bx-\br,\br,t)]\cdot    \delta\ol{u}_\ell (-\br;\bx,t) \cdot \tau_\ell (\bx,t)\rmd t \rmd x \rmd r\\ 
&\ \ \  - \int_0^T \int_{\mathbb{T}^d\times\mathbb{T}^d}\nabla_r \cdot \Big( \delta\ol{u}_\ell (\br;\bx,t) \cdot \tau_\ell (\bx,t)   \varphi(\bx,\br,t) -  \delta\ol{u}_\ell (-\br;\bx,t) \cdot \tau_\ell (\bx,t)   \varphi(\bx-\br,\br,t) \Big)\rmd t \rmd x dr. \label{line3}
\end{align}
The two terms in \eqref{line3} vanish by the divergence theorem since the test function $\psi_R$ has compact support. The terms in \eqref{line2}, \eqref{line2.5} easily are seen to vanish as $\ell\to 0$ for any $R> 0$ since, using the estimate \eqref{CGcumbound} for the cumulant $\tau_\ell$, we have
\begin{align*}
\left|\int_0^T\int_{\mathbb{T}^d\times\mathbb{T}^d} \nabla_r \varphi(\bx,\br,t) \cdot \delta \ol{u}_\ell (\br;\bx,t) \cdot \tau_\ell (\bx,t) \rmd t \rmd x \rmd r\right| \leq &\ \|G\|_1 \|\nabla_r \varphi \|_\infty \times \\
& \quad \quad\quad\quad  \sup_{|\br|\leq R}\| \delta u(\br;\cdot)\|_3\sup_{|\br|\leq \ell}\|\delta \bu (\br;\cdot)\|_3^2,\\
\left| \int_0^T \int_{\mathbb{T}^d\times\mathbb{T}^d}  \nabla_r [\varphi(\bx-\br,\br,t)]\cdot    \delta\ol{u}_\ell (-\br;\bx,t) \cdot \tau_\ell (\bx,t)\rmd t \rmd x \rmd r\right| &\leq  \ \|G\|_1 \max\{\|\nabla_x \varphi \|_\infty,\|\nabla_r \psi_R \|_\infty\}\times \\
& \quad \quad\quad\quad \sup_{|\br|\leq R}\|\delta {\bu} (\br;\cdot)\|_3\sup_{|\br|\leq \ell}\|\delta \bu (\br;\cdot)\|_3^2,
\end{align*}
which vanish again as $\ell\to 0$ due to the strong continuity of shifts in $L^3$.   The consideration of \S \ref{sec:3danom} apply and  $\Pi_\ell[u]$ converge in the sense of distributions to $\Pi[u]$ as $\ell\to 0$.  Thus
 \begin{align} \label{twoPi}
\langle \nabla_x\delta   \ol{u}_\ell(\br;\bx,t) : \delta \tau_\ell (\br;\bx,t) \rangle_\vphi \ \xrightarrow{\ell\to 0}\   \langle\Pi[\bu](\bx+\br,t) \rangle_\vphi+  \langle \Pi[\bu](\bx,t) \rangle_\phi.
\end{align}
Finally we analyze the first term above in the limit of $R\to 0$. Since $\psi_R$ approximates the identity, we have
 \begin{align}\nonumber
  \langle \Pi[\bu](\bx+\br,t) \rangle_\vphi &=  \int_0^T \int_{\mathbb{T}^d\times\mathbb{T}^d}\Pi[\bu](\bx,t)\  \phi(\bx- \br)\psi_R(\br)\rmd t \rmd x \rmd r \\
  &=  \int_0^T\int_{\mathbb{T}^d} \Pi[\bu](\bx,t)\ \psi_R*\phi(\bx)\rmd t \rmd x. \label{PiMol}
\end{align}
Since $\psi_R,\phi\in D(\mathbb{T}^d) = C_0^\infty(\mathbb{T}^d)$, then in the limit of $R\to 0$, we have that $\psi_R*\phi\to \phi $ in the standard Fr\'{e}chet topology on test functions. Since the distribution $\Pi\in D'(\mathbb{T}^d)$ is, by definition,  a continuous linear functional on $D(\mathbb{T}^d)$ we have from Eqns. \eqref{twoPi} and \eqref{PiMol} that
 \begin{align}\label{lastEqnproof}
-\langle \nabla_x\delta   \ol{u}_\ell(\br;\bx,t) : \delta \tau_\ell (\br;\bx,t) \rangle_\vphi \ \xrightarrow{R,\ell\to 0}\  - 2 \langle \Pi[\bu]\rangle_\phi
\end{align}
as claimed.

\noindent \emph{Proof of Lemma \ref{lemma} (ii)}. The proof is nearly identical to that of part (i), now using also the strong convergence assumption that $u^\nu\to u$ in $L^3(0,T;L^3(\mathbb{T}^d))$ and that $f^\nu\in L^2(0,T;L^2(\mathbb{T}^d))$ with norms uniformly bounded in $\nu$.    We highlight here only the most different parts of the proof. The Eulerian acceleration increment \eqref{filtAcc} from the mollified Navier-Stokes Equations is
\begin{align}
 \delta {a}^{\ell,\nu}(r;x) \equiv -\nabla_x \delta\ol{(p^{\nu} )}_\ell(r;x) + \nu \Delta_x \delta \ol{(u^{\nu} )}_\ell(r;x) + \delta \ol{(f^{\nu} )}_\ell(r;x) -\nabla_x\cdot \delta \tau_\ell^\nu (r;x)
\end{align}
 where $\tau_\ell^\nu := \tau_\ell (u^\nu,u^\nu)$.  From \eqref{EulRepRelVelo}, we have that
\begin{align*}
\frac{1}{2} \frac{\rmd}{\rmd \tau} \langle |\delta_r {v}^{\ell}(\tau;x,t)|^2\rangle_{\vphi}\Big|_{\tau=0}  =& -\langle \delta  \ol{(u^{\nu} )}_\ell(\br;\bx) \cdot\nabla_x   \delta\ol{(p^{\nu} )}_\ell(r;x) \rangle_\vphi+\nu \langle\delta \ol{(u^{\nu} )}_\ell(\br;\bx) \cdot\Delta_x \delta\ol{(u^{\nu} )}_\ell(\br;\bx) \rangle_\vphi\\
  &+\langle \delta \ol{(u^{\nu} )}_\ell(\br;\bx) \cdot \delta \ol{(f^{\nu} )}_\ell(\br;\bx) \rangle_\vphi- \langle  \delta \ol{(u^{\nu} )}_\ell(\br;\bx) \cdot\nabla_x\cdot \delta \tau_\ell^\nu (\br;\bx) \rangle_\vphi.
\end{align*}
The only new term involves the viscous friction.  By Young's inequality for convolutions, this term is bounded point-wise in $x,t$ and $r$ by
\begin{align}
\nu |\delta \ol{(u^{\nu} )}_\ell(\br;\bx,t) \cdot\Delta_x \delta\ol{(u^{\nu} )}_\ell(\br;\bx,t) | \lesssim \frac{\nu}{\ell^2} \|G\|_1 \|\Delta G\|_1\|\phi\|_{\infty} \|\bu^\nu(t)\|_2^2\ \xrightarrow{\nu\to 0}\ 0 
\end{align}
since $\bu^\nu$ is uniformly bounded $L^\infty(0,T;L^2(\mathbb{T}^d))$.    For the pressure-work term, we need only that $u^\nu \to u$ strongly in $L^3(0,T;L^3(\mathbb{T}^d))$ implies $p^\nu \to p$ strongly in $L^{3/2}(0,T;L^{3/2}(\mathbb{T}^d))$, which follows from strong continuity of Calderon-Zygmund operators in $L^p$ for $1<p<\infty$.  With these strong convergence statements, the estimates for the remaining terms follow by identical arguments to those appearing in the proof of part (i).    Finally, under our assumptions, the considerations of \S \ref{sec:3danom} apply and the flux distribution $\Pi$ is identified with the viscous dissipation anomaly via \eqref{PiEps}. Therefore, the only non-vanishing term in the end is
 \begin{align}
-\lim_{R\to 0}\lim_{\ell\to 0} \lim_{\nu\to0}\  \langle \nabla_x\delta \ol{(\bu^\nu)}_\ell(\br;\bx) : \delta \tau_\ell^\nu (\br;\bx) \rangle_\vphi\ =\    -2 \langle \varepsilon[\bu] \rangle_\phi. 
\end{align}

\noindent \emph{Proof of Lemma \ref{lemma} (iii)} Once again, the proof is nearly identical to that of part (i),(ii), now with strong convergence assumption that $u^{k_f}\to u$ in $L^3(0,T;L^3(\mathbb{T}^2))$.  The Eulerian acceleration increment \eqref{filtAcc} from the mollified forced Euler equations is
\begin{align}
 \delta {a}^{\ell,k_f}(r;x) \equiv -\nabla_x \delta\ol{(p^{k_f} )}_\ell(r;x) + \delta \ol{(f^{k_f} )}_\ell(r;x) -\nabla_x\cdot \delta \tau_\ell^{k_f} (r;x).
\end{align}
Estimating all the terms above accept the forcing follows easily from our previous arguments.  The forcing term is treated differently than in part (i) and (ii) where it was assumed that $\|f^\nu\|_{L^2(0,T;L^2(\mathbb{T}^2))}$ where uniformly bounded in $\nu$.  For $d=2$ in the setting under consideration, by assumption $f^{k_f}$ does not have $L^2(0,T;L^2(\mathbb{T}^2)$ norms bounded uniformly in $k_f$, since if they were, there could be no energy conservation anomaly arising from the forcing (see Remark \ref{forcingRem} below).  Nevertheless, its contribution to \eqref{EulRepRelVelo}, $\langle \delta \ol{(u^{k_f} )}_\ell(\br;\bx) \cdot \delta \ol{(f^{k_f} )}_\ell(\br;\bx) \rangle_\vphi$, vanishes in the limit $k_f\to \infty$ at fixed $\ell$ since 
\be
|\langle \delta \ol{(u^{k_f} )}_\ell(\br;\bx) \cdot \delta \ol{(f^{k_f} )}_\ell(\br;\bx) \rangle_\vphi| \leq  C\|\varphi\|_\infty \|u^{k_f}\|_2 \| \ol{(f^{k_f})}_\ell \|_2 \xrightarrow{k_f\to \infty}  0
\ee
since $\|u^{k_f}\|_2$ are uniformly bounded and $\ol{(f^{k_f})}_\ell\xrightarrow{k_f\to \infty}  0$ uniformly in $x\in \mathbb{T}^2$ since $f^{k_f}$ vanishes in the sense of distributions as $k_f\to \infty$.  

This convergence can be seen more directly if we assume that the force has compact spectral support with $L^1$ norms bounded by a power of $k_f$, as in Proposition \ref{forcingProp}. In that case, we have
\be
\ol{(f^{k_f})}_\ell = \int_{\mathbb{T}^d} G_\ell(\br) f^{k_f}(\bx+\br)  \rmd r  =   \int_{\mathbb{T}^d}  \bold{P}_{k_f}[G_\ell](\br) f^{k_f}(\bx+\br)  \rmd r ,
\ee
where $\bold{P}_{k_f}$ is the projection onto the spectral support of $f^{k_f}$.  Thus, 
\be
|\ol{(f^{k_f})}_\ell| \leq \| \bold{P}_{k_f}[G_\ell]\|_\infty \|f^{k_f}\|_1 \leq Ck_f^N \sum_{k \ \in\   {\rm supp}( \widehat{f^{k_f}})} |\widehat{G}_\ell(k)|\xrightarrow{k_f\to \infty} 0
\ee
since $ \|f^{k_f}\|_1\leq C k_f^N$, $N>0$ by assumption and $|\widehat{G}_\ell(k)|$ decays faster than any polynomial (see proof of Proposition \ref{forcingProp}).  Thus, we obtain convergence -- uniformly in $\bx$ -- of the mollified force to zero as $k_f\to \infty$. If the force does not have compact spectral support, but has Fourier-transform `concentrated' about $k_f$, the same argument can be modified and applied so long as there is sufficiently rapid decay whenever $||k|-k_f|\gg 1$. 

One final modification of the previous proof; the considerations of  \S \ref{sec:2danom} apply and the distributional flux anomaly $\Pi$ is identified with anomalous forcing input $-I$ in the limit $k_f\to\infty$, see Eq. \ref{PiForce}.  Therefore
 \begin{align}
-\lim_{R\to 0}\lim_{\ell\to 0} \lim_{k_f\to\infty} \langle \nabla_x\delta \ol{(\bu^{k_f})}_\ell(\br;\bx,t) : \delta \tau_\ell^{k_f} (\br;\bx,t) \rangle_\vphi\ =\  2 \langle I[\bu] \rangle_\phi.
\end{align}
 \end{proof}

\begin{proof}[Proof of Proposition \ref{forcingProp}]
For $k_f<\infty$, we assume that $ {\rm supp}( \widehat{f^{k_f}}) \subseteq S(k_f)$ where the set $S(k_f) = \{ k \  | \ k_f/2\leq |k|\leq 2  k_f\}$ is a shell in wavenumber space. 
Note that, for each fixed $k_f$ the forcing is necessarily smooth by this frequency localization assumption.  Thus,  since $f^{k_f}\in L^2(0,T; L^2(\mathbb{T}^2))$,  for any test function $\varphi \in C_0^\infty([0,T]\times \mathbb{T}^2)$ we have
\begin{equation}\label{distforcEq}
\int_0^T \int_{\mathbb{T}^2}  \varphi(\bx,t) f^{k_f}(\bx,t)\rmd t \rmd x= \int_0^T \int_{\mathbb{T}^2}   \bold{P}_{k_f}[\varphi(\bx,t)]f^{k_f}(\bx,t)\rmd t\rmd x,
\end{equation}
where $\bold{P}_{k_f}$ is the projection onto the shell $S(k_f)$ of wavenumber support of the force $f^{k_f}$. 
Since the Fourier transform of the $C^\infty$ function decays faster than any polynomial, i.e. $|\widehat{\varphi}(k,t)|=O(|k|^{-n})$ as $|k|\to \infty$ for any $n\in \mathbb{N}$ and $t\in(0,T)$, we see that
\begin{align} \nonumber
\left|\int_0^T \int_{\mathbb{T}^2} \bold{P}_{k_f}[\varphi(\bx,t)]f^{k_f}(\bx,t) \rmd t\rmd x \right| &\leq   \| \bold{P}_{k_f}\varphi \|_\infty \|f^{k_f}\|_1\leq  C k_f^N  \sum_{k\in S(k_f)} |\widehat{\varphi}(k)| \\
& \leq   C \sum_{k\in S(k_f)} (k_f/k)^N k^{N-n} \  \xrightarrow{k_f\to \infty} \ 0
\end{align}
since $n>N$ and $k_f/k\in[1/2, 2]$ on  $k\in S(k_f)$.
By the equality \eqref{distforcEq}, we have that $f^{k_f}\to 0$ in the sense of distributions as $k_f\to \infty$. 
\end{proof}

\begin{rem}\label{forcingRem}
In  Proposition \ref{forcingProp},  we do not assume that $L^2(0,T;L^2(\mathbb{T}^2))$  norms of $f^{k_f}$ uniformly bounded $k_f$.  In fact, if this were the case, the forcing would be unable to sustain an inverse cascade asymptotically!  To see this, we additionally assume that $\bu^{k_f}\to \bu$ strongly in $L^2(0,T;L^2(\mathbb{T}^2))$ as ${k_f\to \infty}$.  Then
\begin{align*}
\int_{\mathbb{T}^2} \varphi(\bx) \bu^{k_f}(\bx) f^{k_f}(\bx)\rmd x&= \int_{\mathbb{T}^2} \varphi(\bx) \bu(\bx) f^{k_f}(\bx) \rmd x+ \int_{\mathbb{T}^2} \varphi(\bx)\left(\bu^{k_f}(\bx) -\bu(\bx)\right)f^{k_f}(\bx) \rmd x\\
&= \int_{\mathbb{T}^2}\bold{P}_{k_f}\left[ \varphi(\bx) \bu(\bx)\right] f^{k_f}(\bx) \rmd x+ \int_{\mathbb{T}^2} \varphi(\bx)\left(\bu^{k_f}(\bx) -\bu(\bx)\right)f^{k_f}(\bx)\rmd x
\end{align*}
 for any test function $\varphi \in C_0^\infty([0,T]\times \mathbb{T}^2)$. The second integral vanishes due to the strong convergence $\bu^{k_f}$ as ${k_f\to \infty}$ as can be seen from
\be
\left|\int_{\mathbb{T}^2}\varphi(\bx)\left(\bu^{k_f}(\bx) -\bu(\bx)\right)f^{k_f}(\bx) \rmd x\right|\leq \|\varphi\|_\infty \|\bu^{k_f}-\bu\|_2 \|f^{k_f}\|_2\xrightarrow{k_f\to \infty} 0.
\ee
On the other hand, the first term also vanishes because
\be
\left|\int_{\mathbb{T}^2} \bold{P}_{k_f}\left[ \varphi(\bx) \bu(\bx)\right] f^{k_f}(\bx)\rmd x \right|\leq \|\bold{P}_{k_f}\left[ \varphi\bu\right] \|_2 \| f^{k_f}\|_2.
\ee
A simple application of H\"{o}lder's inequality shows $\varphi \bu\in L^2$, so that $\|\bold{P}_{k_f}\left[ \varphi\bu\right] \|_2\xrightarrow{k_f\to \infty} 0$. Thus, the power input from the force will vanish distributionally if  $\|f^{k_f}\|_2$ is bounded uniformly in $k_f$. I am grateful to G. Eyink for this observation.
\end{rem}

\subsection*{Acknowledgments}  I am grateful to G. Eyink for numerous helpful comments and discussions.  I would also like to thank P. Constantin,  N. Constantinou, A. Frishman, P. Isett, H.Q. Nguyen, V. Vicol and M. Wilczek, as well as the anonymous referees, for comments that greatly improved the paper. Research of the author is supported by NSF-DMS grant 1703997.



\end{document}